\documentclass[reqno, 12pt]{amsart}
\makeatletter
\let\origsection=\section \def\section{\@ifstar{\origsection*}{\mysection}} 
\def\mysection{\@startsection{section}{1}\z@{.7\linespacing\@plus\linespacing}{.5\linespacing}{\normalfont\scshape\centering\S}}
\makeatother        

\usepackage{amsmath,amssymb,amsthm}
\usepackage{mathrsfs}
\usepackage{mathabx}\changenotsign

\usepackage{xcolor}
\usepackage[backref]{hyperref}
\hypersetup{
    colorlinks,
    linkcolor={red!60!black},
    citecolor={green!60!black},
    urlcolor={blue!60!black}
}

\usepackage[abbrev,msc-links,backrefs]{amsrefs}
\usepackage[T1]{fontenc}
\usepackage{lmodern}

\usepackage[english]{babel}

\usepackage{geometry}
\usepackage{enumitem}
\def\rmlabel{\upshape({\itshape \roman*\,})}
\def\RMlabel{\upshape(\Roman*)}
\def\alabel{\upshape({\itshape \alph*\,})}

\let\polishlcross=\l
\def\l{\ifmmode\ell\else\polishlcross\fi}

\let\emptyset=\varnothing
\let\setminus=\smallsetminus

\theoremstyle{plain}
\newtheorem{theorem}             {Theorem}[section]
\newtheorem{lemma}     [theorem] {Lemma}           
\newtheorem{fact}[theorem] {Fact}   
   
\newtheorem{corollary}[theorem] {Corollary}

\newtheorem{definition}[theorem] {Definition}

\def\beq{\begin{equation}}\def\eeq{\end{equation}}
\def\beqn{\begin{eqnarray}}\def\eeqn{\end{eqnarray}}

\def\eps{\varepsilon}

\def\KE{\mathop{\text{\rm KE}}\nolimits}
\def\int{\mathop{\text{\rm int}}\nolimits}

\begin{document}

\title{Ramsey numbers for bipartite graphs with small bandwidth}

\author[G.~O.~Mota]{Guilherme O. Mota}
\address{Instituto de Matem\'atica e Estat\'{\i}stica, Universidade de S\~ao Paulo, S\~ao Paulo, Brazil}
\email{mota@ime.usp.br}

\author[G.~N.~S\'ark\"ozy]{G\'abor N. S\'ark\"ozy}
\address{Computer Science Department, Worcester Polytechnic Institute, Worcester, USA and
	Alfr\'ed R\'enyi Institute of Mathematics, Hungarian Academy of Sciences, Budapest, Hungary}
\email{gsarkozy@cs.wpi.edu}

\author[M.~Schacht]{Mathias Schacht}
\address{Fachbereich Mathematik, Universit\"at Hamburg, Hamburg, Germany}
\email{schacht@math.uni-hamburg.de}

\author[A.~Taraz]{Anusch Taraz}
\address{Institut f\"ur Mathematik, Technische Universit\"at Hamburg--Harburg, Hamburg, Germany}
\email{taraz@tuhh.de}

\thanks{G.~O.~Mota was supported by FAPESP (2009/06294-0 and 2012/00036-2) and he gratefully 
	acknowledges the support of NUMEC/USP, Project MaCLinC/USP. G.N.~S\'ark\"ozy was supported 
	in part by the NSF Grant DMS-0968699 and by OTKA Grant K104373. M.~Schacht was supported 
	by the Heisenberg-Programme of the DFG (grant SCHA 1263/4-1). 
	A.~Taraz was supported in part by DFG grant TA 309/2-2. The cooperation was supported 
	by a joint CAPES-DAAD project (415/ppp-probral/po/D08/11629, Proj.~no.~333/09).}


\begin{abstract}
We estimate Ramsey numbers for bipartite graphs with small bandwidth and bounded 
maximum degree. In particular we determine asymptotically the two
and three color Ramsey numbers for grid graphs. More generally, we determine asymptotically the two color Ramsey number 
for bipartite graphs with small bandwidth and 
bounded maximum degree and the three color
Ramsey number for such graphs with the additional assumption that the bipartite graph is balanced.
\end{abstract}

\maketitle

\section{Introduction}
For graphs $G_1, \ldots , G_r$, the Ramsey number $R(G_1, \ldots , G_r)$ is 
the smallest integer~$n$ such that
if the edges of a complete graph $K_n$ are partitioned into $r$
disjoint color classes giving~$r$ graphs $H_1, \ldots , H_r$,
then at least one $H_i$ ($1\leq i \leq r$) contains a subgraph
isomorphic to $G_i$. The existence of such an integer follows from Ramsey's theorem. The number
$R(G_1, \ldots , G_r)$ is called the Ramsey number of the
graphs $G_1, \ldots , G_r$. 
Determining $R(G_1, G_2, \ldots , G_r)$ for general graphs appears to 
be a difficult problem (see, e.g.,~\cite{GrRoSp90} or~\cite{Ra94}). 
For $r=2$, a well-known theorem of
Gerencs\'er and Gy\'arf\'as~\cite{GeGy67} states that
$$R(P_n, P_n) = \left\lfloor \frac{3n-2}{2} \right\rfloor,$$
where $P_n$ denotes the path with $n\geq 2$ vertices.
In~\cite{HaLuTi02} more general trees were considered. For a tree $T$, we
write $t_1$ and $t_2$, with $t_2\geq t_1$, for the sizes of the vertex
classes of~$T$ as a bipartite graph. Note that $R(T,T)\geq 2t_1+t_2-1$, since the following edge-coloring of~$K_{2t_1+t_2-2}$ has no monochromatic 
copy of $T$. Partition the
vertices into two classes $V_1$ and~$V_2$ such that $|V_1|=t_1 - 1$ and $|V_2|=t_1+t_2-1$, then use color ``red'' for all edges inside the two classes 
and use color ``blue'' for all edges between the
classes. A similar edge-coloring of $K_{2t_2 -2}$ with two classes both of size $t_2-1$ shows that 
$R(T,T)\geq 2t_2-1$.
Thus
\begin{equation}\label{eq:ramseyTwoTree}
R(T,T)\geq \max \{ 2t_1+t_2, 2t_2 \} - 1.
\end{equation}
Haxell, {\L}uczak and Tingley provided in~\cite{HaLuTi02} an asymptotically matching upper bound for trees $T$ with $\Delta(T)=o(t_2)$.

We partially extend this to bipartite graphs with
small bandwidth and a more restrictive maximum degree
condition. A graph $H=(W,E_H)$ is said to have {\em bandwidth} at most $b$,
if there exists a labelling of the vertices by numbers $1, \ldots , n$ such that for every edge $ij\in E_H$ we have
$|i-j|\leq b$. We focus our attention on the following class of bipartite graphs.

\begin{definition}
A bipartite graph $H$ is a \emph{$(\beta,\Delta)$-graph} if its 
bandwidth is at most~$\beta |V(H)|$ and its maximum degree is at most
$\Delta$. We say that $H$ is a \emph{balanced
$(\beta,\Delta)$-graph} if it has a proper $2$-coloring $\chi\colon V(H)
\to [2]$ such that  ${\big||\chi^{-1}(1)|-|\chi^{-1}(2)|\big|
\leq \beta |\chi^{-1}(2)|}$.
\end{definition}

For example, it was shown in~\cite{BoPrTaWu10} that sufficiently large planar graphs with maximum degree at most 
$\Delta$ are $(\beta, \Delta)$-graphs 
for any fixed~$\beta>0$.  Our first theorem is
an analogue of the
result in~\cite{HaLuTi02} for $(\beta, \Delta)$-graphs.

\begin{theorem}
\label{thm:twocolor}
For every $\gamma>0$ and natural number $\Delta$, there exist a
constant $\beta>0$ and natural number $n_0$ such that for every
$(\beta,\Delta)$-graph $H$ on ${n\geq n_0}$ vertices with a proper
2-coloring $\chi: V(H) \to [2]$ where $t_1=|\chi^{-1}(1)|$ and
${t_2= |\chi^{-1}(2)|}$, with ${t_1\leq t_2}$, we have
$$
R(H,H) \leq (1+\gamma) \max \{2t_1+t_2, 2t_2 \}.
$$
\end{theorem}
For more recent results on the Ramsey number of graphs 
of higher chromatic number and sublinear bandwidth, we refer the reader 
to the work of Allen, Brightwell and Skokan~\cite{AlBrSk13}.

For $r\geq 3$ colors less is known about Ramsey numbers. Let $T$ be a tree and consider $t_1$ and $t_2$, with $t_1\leq 
t_2$, the sizes of the
vertex classes of $T$ as a bipartite graph. For $r=3$ colors the following construction gives a lower bound for 
$R(T,T,T)$. Partition the vertices of $K_{t_1+3t_2-3}$ into 
four classes, one special class $V_0$ with
$|V_0|=t_1$ and three classes $V_1$, $V_2$ and $V_3$ of size $t_2-1$. The color for edges inside $V_0$ is arbitrary. Use 
color $i$ inside the classes 
$V_i$ and color $i$ between $V_i$ and $V_0$ for
$1\leq i\leq 3$. Finally, use color $k\in[3]\setminus\{i,j\}$ for edges between the classes $V_i$ and $V_j$ for $1\leq 
i<j\leq 3$. It is easy to check 
that this coloring yields no monochromatic
copy of $T$. Thus
\begin{equation}\label{eq:ramseyThreeTree}
R(T,T,T)\geq t_1 + 3t_2 - 2.
\end{equation}

Proving a conjecture of Faudree and Schelp~\cite{FaSc75},  it was shown in~\cite{GyRuSaSz07b} that this construction is 
optimal for large paths, i.e., 
for sufficiently large $n$ we have
$$R(P_n, P_n, P_n) = \left\{
\begin{array}{l}
2n - 1 \; \; \mbox{for odd} \; n,\\
2n - 2 \; \; \mbox{for even} \; n.
\end{array}\right.$$
Asymptotically this was also proved independently by Figaj and
{\L}uczak \cite{FiLu07}. Be\-ne\-vi\-des and Skokan~\cite{BeSk09} proved that $R(C_n,C_n,C_n)=2n$ for sufficiently large 
even $n$. Our second result 
extend the two above ones asymptotically to balanced $(\beta, \Delta)$-graphs.
\begin{theorem}\label{thm:threecolor}
For every $\gamma>0$ and every natural number $\Delta$, there exist a
constant $\beta>0$ and natural number $n_0$ such that for every
balanced $(\beta,\Delta)$-graph $H$ on $n\geq n_0$ vertices we
have
$$
R(H,H,H) \leq (2+\gamma) n.
$$
\end{theorem}
In particular, Theorems \ref{thm:twocolor} and \ref{thm:threecolor} give the asymptotics for two and three color Ramsey 
numbers of grid graphs. The 
$2$-dimensional grid graph $G_{a,b}$ is
the graph with vertex set ${V=[a]\times[b]}$ and there is an edge between two vertices if they are equal in one 
coordinate
and consecutive in the other. Note that any grid graph $G_{a,b}$ on $ab$ vertices has bandwidth at 
most $\min\{a,b\}$ and satisfies $\Delta(G)\leq 4$. Moreover, $G_{a,b}$ is a balanced
$(\beta,4)$-graph for any fixed $\beta>0$ and sufficiently large $ab$. Consequently, Theorems \ref{thm:twocolor} and 
\ref{thm:threecolor} combined 
with \eqref{eq:ramseyTwoTree} and
\eqref{eq:ramseyThreeTree} give the following corollary. 

\begin{corollary}\label{cor:grid}
For grid graphs $G_{a,b}$ we have
\begin{align*}
R(G_{a,b},G_{a,b}) &= \big(3/2 + o(1)\big)ab\\
\intertext{and}
R(G_{a,b},G_{a,b},G_{a,b}) &= \big(2+o(1)\big)ab,
\end{align*}
where $o(1)$ tends to 0 as $ab\to\infty$.
\end{corollary}
We remark that similar bounds follow for bipartite planar graphs with bounded degree and grids of higher dimension.

This paper is organized as follows. We first give the necessary tools in Section~\ref{sec:aux} and then present a 
detailed proof of 
Theorem~\ref{thm:threecolor} in Section~\ref{sec:main}. The proof of Theorem~\ref{thm:twocolor} is very similar and here 
we only present an outline, in 
Section~\ref{sec:two}.

\section{Auxiliary results}\label{sec:aux}

The main purpose of this section is to present the tools for the proof of Theorem~\ref{thm:threecolor}. A main tool in 
the proof is Szemer\'edi's 
Regularity Lemma~\cite{Sz75}. We discuss the
Regularity Method in Section~\ref{subsec:reg}. In Sections~\ref{subsec:blowtree} and~\ref{subsec:balBlocks} we give some 
results that allow us to make use of the regularity method.

\subsection{The Regularity Method}\label{subsec:reg}

Given an graph $G$ on $n$ vertices, the \emph{density} of $G$ is given by $d_G=e(G)/{n\choose 2}$. Furthermore, if $A$, 
$B\subset V(G)$ are non-empty 
and disjoint, we denote by
$e_G(A,B)$ the number of edges of $G$ with one endpoint in $A$ and
the other in $B$ and
\begin{equation*}
d_G(A,B)=\frac{e_G(A,B)}{|A||B|}
\end{equation*}
is the \emph{density} of $G$ between $A$ and $B$. 

The bipartite graph $G=(A,B;E)$ is called $\eps$\emph{-regular} if
for all $X\subset A$, $Y\subset B$ with $|X|>\eps|A|$ and $|Y|>\eps|B|$ we have
\begin{equation*}
|d_G(X,Y)-d_G(A,B)|<\eps. 
\end{equation*}
We say that $G$ is $(\eps,d)$\emph{-regular} if it is
$\eps$-regular and $d_G(A,B)\geq d$. An $\eps$-regular bipartite graph $(A,B;E)$ is called 
$(\eps,d)$\emph{-super-regular} if we have ${\deg_G(a)>d|B|}$ for 
all $a\in A$ and also ${\deg_G(b)>d|A|}$ for all $b\in B$.

For a graph $G=(V,E)$, a partition $(V_i)_{i\in [s]}$ of $V$ is said to be $(\eps,d)$\emph{-regular} (resp.\ 
\emph{super-regular}) \emph{on a graph} 
$R$ with vertex set contained in $[s]$ if the
bipartite subgraph of $G$ induced by the pair $\{V_i,V_j\}$ is $(\eps,d)$-regular (resp.\ super-regular) whenever $ij\in 
E(R)$.
We say that a graph $R$ on vertex set $[s]$ is the $(\varepsilon,d)$-\emph{reduced graph} of $(V_i)_{i\in [s]}$ (or of 
$G$) if $ij$ is an edge of $R$ 
if and only if 
the bipartite graph defined by the pair $\{V_i,V_j\}$ is $(\varepsilon,d)$-regular in $G$. 
The proof of Theorem
\ref{thm:threecolor} is based on the following three color version
of the Regularity Lemma.
\begin{lemma}[Regularity Lemma]\label{lemma:Regularity}
For every $\eps>0$ and every integer $k_0>0$ there is a
positive integer $K_0(\eps, k_0)$ such that for $n\geq K_0$ the
following holds. For all graphs $G_1$, $G_2$ and $G_3$ with
$V(G_1)=V(G_2)=V(G_3)=V$, $|V|=n$, there is a partition of $V$
into $k+1$ classes $V=V_0,V_1,V_2,\dots,V_k$ 
such that
\begin{enumerate}[label=\rmlabel]
\item $k_0\leq k\leq K_0$,
\item $|V_1|=|V_2|=\dots=|V_k|$,
\item $|V_0|<\eps n$,
\item apart from up to at most $\eps {k \choose 2}$ exceptional pairs, the pairs $\{V_i,V_j\}$ are $\eps$-regular in~$G_1$,~$G_2$ and~$G_3$.
\end{enumerate}
\end{lemma}

For extensive surveys on the Regularity Lemma and its
applications see \cites{KoSi96,KoShSiSz02}.
A key component of the regularity method is the Blow-up
Lemma \cite{KoSaSz97} (see also \cites{KoSaSz98, RoRu99,RoRuTa99} for alternative proofs). This lemma guarantees that
bipartite spanning subgraphs of bounded degree can be embedded
into super-regular pairs. In fact, the statement is more general
and allows the embedding of $r$-chromatic graphs into the union of
$r$ vertex classes that form ${r \choose 2}$ super-regular pairs.

Here we will use a version of the Blow-up lemma that allows us to embed graphs $H$ of bounded-degree in a graph $G$ when 
$G$ and $H$ have 
``compatible'' partitions, in the sense explained in the
definition below. In our proof we will embed $H$ in parts, considering a partition of a monochromatic subgraph $G$ of 
$K_N$ with corresponding 
reduced graph containing a tree~$T$
that contains a ``large'' matching $M$, where the bipartite subgraphs of~$G$ corresponding to the matching edges are 
super-regular pairs.

\begin{definition}\label{def:compatiblePartitions}
Suppose $H\!=\!(W, E_H)$ is a graph, $T\!=\!([s], E_T)$ is a tree, and $M\!=\!([s], E_{M})$ is a subgraph of $T$ where $E_M$ is 
a matching. Given a 
partition $(W_i)_{i\in[s]}$ of $W$, let $U_i$, for $i\in[s]$, be the set of
vertices in $W_i$, with neighbors in some $W_j$ with $ij \in E_{T}\setminus E_{M}$ and set~$U= \bigcup U_i$ and $U'_i= 
N_H(U)\cap(W_i\setminus U)$.

We say that $(W_i)_{i\in[s]}$ is \emph{$(\varepsilon,T,M)$-compatible} with a vertex partition $(V_i)_{i\in[s]}$ of
some graph $G=(V,E)$ if the
following holds:
\begin{enumerate}[label=\RMlabel]
 \item\label{d:cPi} $xy\in E_H$ for $x\in W_i$ and $y\in W_j$ implies $ij\in E_T$ for all $i, j \in [s]$,
 \item\label{d:cPii} $|W_i|\leq |V_i|$ for all $i\in [s]$,
 \item\label{d:cPiii} $|U_i|\leq \varepsilon |V_i|$ for all $i\in [s]$,
 \item\label{d:cPiv} $|U'_i|,|U'_j|\leq \eps \min\{|V_i|,|V_j|\}$ for all $ij\in E_M$.
\end{enumerate}
\end{definition}

We remark that for connected graphs $H$ and for every vertex $i$ of $T$ which is not covered by $M$ we have $U_i=W_i$ 
and $U'_i=\emptyset$.

The following corollary of the Blow-up Lemma (see \cite{Bo09}) asserts that in the setup of Definition 
\ref{def:compatiblePartitions} graphs $H$ of 
bounded degree can be embedded into $G$,
if $G$ admits a partition being sufficiently regular on $T$ and super-regular on $M$.

\begin{lemma}[Embedding Lemma \cites{Bo09,BoHeTa10}]\label{lemma:GeneralEmbedding}
For all $d, \Delta > 0 $ there is a 
${\varepsilon = \varepsilon(d, \Delta) > 0}$ such that the following holds. 

Let $G = (V, E)$ be an $N$-vertex
graph that has a partition $(V_i)_{i\in[s]}$ of $V$ with $(\varepsilon, d)$-reduced graph $T$ on~$[s]$ which is
$(\varepsilon, d)$-super-regular on a graph $M \subset T$. Further, let $H = (W, E_H)$ be an $n$-vertex graph with 
maximum degree
$\Delta(H) \leq \Delta$ and $n\leq N$ that has a vertex partition $(W_i)_{i\in[s]}$ of $W$ which is 
$(\varepsilon,T,M)$-compatible with
$(V_i)_{i\in[s]}$. Then $H\subset G$.
\end{lemma}

We close this section with two simple facts. They follow easily from the definitions of regular and super-regular pairs.

\begin{fact}\label{fact:Slicing}
Let $B=(V_1,V_2;E)$ be an $\varepsilon$-regular bipartite graph, let $\alpha>\varepsilon$, and
let $V_1'\subset V_1$ and ${V_2'\subset V_2}$ be subsets with 
$|V_1'|\geq \alpha |V_1|$ and 
$|V_2'|\geq \alpha |V_2|$. 
Then for $\varepsilon'=\max\{\varepsilon/\alpha,2\varepsilon\}$
the graph ${B'=\big(V_1',V_2';E_B(V_1',V_2')\big)}$ is 
$\varepsilon'$-regular with
${|d_B(V_1,V_2)-d_{B'}(V_1',V_2')|<\varepsilon}$.
\end{fact}

\begin{fact}\label{fact:SuperSlicing}
Consider a graph $G=(V,E)$ with an $(\eps,d)$-regular partition $(V_i)_{i\in[s]}$ of $V$ with $|V_i|=m$ for $1\leq 
i\leq s$. Let $T$ be a graph on vertex set $[s]$ contained in the corresponding
$(\eps,d)$-reduced graph of $(V_i)_{i\in[s]}$ and let $M$ be a matching contained in~$T$. Then for each vertex $i$ of 
$M$, the associated set $V_i$ 
in $G$ contains a
subset $V_i'$ of size $(1-\varepsilon r)m$ such that for every edge $ij$ of $M$ the bipartite graph 
$(V_i',V_j';E_G(V_i',V_j'))$ is 
$(\varepsilon/(1-\varepsilon r),
d-(1+r)\varepsilon)$-super-regular.
\end{fact}

\subsection{Regular blow-up of a tree}\label{subsec:blowtree}

In this section we show, in Lemma~\ref{lemma:FindTheNiceTree}, that for any coloring of $E(K_N)$ there exists a dense, 
regular, monochromatic subgraph of $K_{N}$ with some structural 
properties that allow us to embed
$H$ into this subgraph. Here the
notion of a connected matching in the reduced graph (originating in \cite{Lu99}, see also 
\cites{FiLu07,GyRuSaSz07b,GyRuSaSz07,HaLuPeRoRuSiSk06}) plays 
a central role. A {\em connected matching}
in a graph $R$ is a matching~$M$ such that all edges of~$M$ are in
the same connected component of $R$. The following lemma, proved in~\cite{FiLu07}, states that in
a $3$-colored almost complete graph we can always find a connected matching that
covers almost half of the vertices and it is contained in a monochromatic tree. 

\begin{lemma}\label{lemma:Matching}
For every $\delta>0$ there exist an $\eps_0 >0$ and a natural
number $k_0$ such that for every $\eps < \eps _0$ and $k\geq
k_0$ and for every 3-edge colored graph $R$ on
$k$ vertices with density at least $(1-\eps)$ there exists a matching $M$ with at least $(1-\delta)k/4$ edges in $R$ 
that is contained in a 
monochromatic tree $T\subset R$.
\end{lemma}

This lemma can be found in a stronger structural form in
\cite{GyRuSaSz07b}. In fact, there it is proved that either
there is a monochromatic connected matching covering more than
half of the vertices, or the graph $R$ is close to one of two
extremal cases. It is not hard to see that in both extremal cases there is a monochromatic connected
matching $M$ of size at least $(1-\delta)|V(R)|/4$. We will also make use of the following simple fact.
\begin{fact}\label{fact:TreeEvenDistance}
If a tree $T$ contains a matching $M$ with $\ell$ edges then the
vertices covered by the matching can be labelled in such way that ${E_M=\{x_i y_i\colon i=1,\dots,\ell\}}$ and $x_i$ 
and $x_j$ are at an even distance 
in $T$ for all
$1\leq i<j\leq \ell$.
\end{fact}
Indeed, consider a proper two-coloring $\chi\colon\, V(T) \to [2]$. Label
those endpoints of the matching edges with $x_i$ that are in
$\chi^{-1}(1)$ and label the other endpoints by $y_i$. Clearly,
the distance in $T$ between any $x_i$ and $x_j$ is even, since
they belong to the same color class.

Given a coloring $\chi_{K_N}\colon\, E(K_{n})\to[3]$, we denote by $G_1$ the spanning subgraph of $K_n$ such that 
$ij\in E(G_1)$ if and only if $\chi_{K_N}(ij)=1$.

\begin{lemma}\label{lemma:FindTheNiceTree}
For every $\gamma> 0$ there exists an $\eps_0>0$ such that for every
$\eps\in(0,\eps_0)$, there exists a natural number $K_0$ such that for
all $N=(2+\gamma)n \geq K_0$ and for every coloring $\chi_{K_N}\colon E(K_{N})\to[3]$, there exist a color (say color 
1), integers $\ell,\ell',k$ with 
$\ell,\ell'\leq k \leq
K_0$ and $\ell\geq (1-\gamma/4)k/4$, a tree $T$ on vertex set
$\{x_1,\dots,x_\ell,y_1,\dots,y_\ell,z_1,\dots,z_{\ell'}\}$ containing a matching $M$ with edge set $E_M=\{x_i 
y_i\colon i=1,\dots,\ell\}$ with an 
even distance in $T$
between any $x_i$ and $x_j$ for all $i$ and $j$, such that there exists a partition $(V_i)_{i\in[k]}$ of $V(K_N)$ such 
that $G_1$ is 
$(\eps,1/3)$-regular on $T$ and ${|V_1|=\ldots=|V_k|\geq
(1-\eps)N/k}$.
\end{lemma}

\begin{proof}
Fix $\gamma>0$ and set $\delta=\gamma/4$. Let $\varepsilon_0$ and $k_0$ be the constants obtained from Lemma~\ref{lemma:Matching} applied with $\delta$. 
Fix $\varepsilon<\varepsilon_0$ and let $K_0$ be
obtained by an 
application of the Regularity Lemma (Lemma \ref{lemma:Regularity}) with parameters $\varepsilon$ and $k_0$. Finally let 
${N=(2+\gamma)n\geq K_0}$ be 
given.

Consider an arbitrary $3$-coloring $\chi_{K_N}\colon E(K_{N})\to[3]$ of the edges of $K_{N}$ and spanning subgraphs 
$G_1$, $G_2$ and $G_3$ of $K_{N}$ 
where $ij\in G_s$ if and only if $\chi_{K_N}(ij)=s$,
for $s=1,2,3$. Owing to the Regularity Lemma, there is a partition $V_0,V_1,\ldots,V_k$ of the vertices of $K_{N}$ such 
that $|V_i|=m\geq 
(1-\varepsilon)N/k$ 
for $1\leq i\leq k$ and more than $(1-\varepsilon){k\choose 2}$ pairs $\{V_i,V_j\}$ for $1\leq i<j\leq k$ are 
$\varepsilon$-regular in $G_1$, 
$G_2$ and $G_3$, where $k_0\leq k\leq
K_0$.

We define the following reduced graph: let $R$ be the graph with vertex set $[k]$, which contains the edge $ij$ if and only if 
$\{V_i,V_j\}$ is 
$\varepsilon$-regular in each of $G_1$, $G_2$ and $G_3$.
Thus, ${|E(R)|\geq (1-\varepsilon){k\choose 2}}$. We know that $R$ is a graph on $k$ vertices with density 
at least $(1-\varepsilon)$. Now
we define a coloring $\chi_R\colon E(R)\to[3]$ of the edges of $R$ such that $\chi_R(i,j)=s$ if $s\in[3]$ is the 
biggest integer in $[3]$ such that 
$|E_{G_s}(V_i,V_j)| \geq |E_{G_r}(V_i,V_j)|$ for $1\leq r\leq 3$,
i.e., the edge~$ij$ receives one of the colors that appears in most edges of $E_{K_N}(V_i,V_j)$ with respect to the 
coloring $\chi_{K_N}$ of 
$E(K_{N})$. If $\chi_R(ij)=s$, then
$|E_{G_s}(V_i,V_j)|\geq |V_i||V_j|/3$.

Since $k\geq k_0$ and the density of~$R$ is at least $(1-\varepsilon)$, by Lemma \ref{lemma:Matching}, 
we know that~$R$ 
contains a monochromatic tree 
$T$ that contains a
matching $M$ of size $\ell\geq(1-\delta)k/4$. Without loss of generality we may assume that the edges of $T$ are 
colored with color~$1$. By 
Fact~\ref{fact:TreeEvenDistance} we can label $M=(\{x_i,y_i\})_i$ such
that $x_i$ and $x_j$ are at even distance in $T$ for every $1\leq i<j\leq \ell$.

Let $\{z_1,\ldots,z_{\ell'}\}$ be the vertices of $T$ that are not covered by edges of the matching~$M$. Since all the 
edges of $T$ are present in $R$ 
we know that, for all $ij\in
E(T)$, the pairs~$\{V_i,V_j\}$ are $\varepsilon$-regular in $G_1$ with $|E_{G_1}(V_i,V_j)|\geq |V_i||V_j|/3$. Thus 
we are done, since we can consider the graph composed of the classes $V_i$ for
every $i\in V(T)$ and with edge set $E_{G_1}(V_i,V_j)$ between every pair.                                            
\end{proof}

\subsection{Balanced intervals}\label{subsec:balBlocks}

By definition, given $\beta>0$ and a natural $\Delta$, a balanced $(\beta,\Delta)$-graph $H$ has a 2-coloring of its 
vertices that uses both colors 
similarly often \emph{in total}, 
but this does not have to be true \emph{locally}. In this section we show how to balance $H$ so that the two colors 
appear in approximately the same 
number of vertices also locally.

Given a graph $H=(W,E)$ with $W=\{w_1,\ldots,w_n\}$, let ${\chi\colon W\to[2]}$ be a $2$-coloring. Define the function 
$C_i$ such that if $W'\subset W$ 
then, for $i=1,2$, we have ${C_i(W')=|\chi^{-1}(i)\cap W'|}$. We
say that $\chi$ is a $\beta$\emph{-balanced} coloring of $W$ if $1-\beta\leq C_1(W)/C_2(W)\leq 1+\beta$. A subset 
$I\subset W$ is called 
\emph{interval} if there exists $p<q$ such that
$I=\{w_p,w_{p+1},\ldots,w_q\}$.
Finally, let~$\ell'$ and $\hat\ell$ be positive integers with $\ell'\leq \hat\ell$ and let 
$\sigma\colon[\ell']\to[\hat\ell]$ be an injection. 
Consider a partition of $W$ into a set of intervals
$\mathcal{I}=\{I_1,\ldots,I_{\hat\ell}\}$. We define $C_i(\mathcal{I},\sigma,a)=\sum_{j=1}^{a} C_i(I_{\sigma(j)})$ for 
$i=1,2$. If it is clear what 
partition we are considering then we write $C_i(\sigma,a)$ for simplicity.

Given a graph $H=(W,E)$, let $c\colon W\to[2]$ be a coloring of $W$ such that $H$ is globally 
balanced. Roughly speaking, the next lemma states that every partition of~$W$ into intervals of almost the same size 
can be rearranged in some way 
that, after the rearrangement, if we remove the ``last'' intervals, then, in the subgraph of $H$ induced by the 
remaining vertices, the difference 
between the number of vertices $w$ with $c(w)=1$ and those $w$ with $c(w)=2$ is ``small''.

\begin{lemma}\label{lemma:balCol}
For every integer $\hat \ell\geq 1$ there exists $n_0$ such that if $H=(W,E)$ is a graph with $W=\{w_1,\ldots,w_n\}$ 
with $n\geq n_0$, then every 
$\beta$-balanced $2$-coloring $\chi$ of $W$ with
$\beta\leq 2/\hat\ell$, and every partition of $W$ into intervals $I_1,\ldots,I_{\hat\ell}$ with sizes 
${|I_1|\leq\ldots \leq|I_{\hat\ell}|\leq |I_1|+1}$ there 
exists a permutation
$\sigma\colon[\hat\ell]\to[\hat\ell]$ such that for every $1\leq i \leq \hat\ell$ we have
\begin{equation*}
\left|C_1(\sigma,i) - C_2(\sigma,i)\right|\leq \frac{n}{\hat\ell}+1.
\end{equation*}
\end{lemma}

\begin{proof}
Fix $\hat\ell\geq 1$ and set $n_0=2{\hat\ell}^3$.  Let $H=(W,E)$ be a graph such that  ${W=\{w_1,\ldots,w_n\}}$ with 
$n\geq n_0$. Fix a $\beta$-balanced 
coloring $\chi$ of
$W$ and a partition of $W$ into intervals $I_1,\ldots,I_{\hat\ell}$ with $|I_1|\leq\ldots \leq|I_{\hat\ell}|\leq 
|I_1|+1$ where $\beta\leq 2/\hat\ell$.

Let us construct the permutation $\sigma$ iteratively. We can take any integer on $[\hat\ell]$ to be~$\sigma(1)$, 
because the size of the 
intervals is at most $n/\hat\ell+1$.
Suppose  $\sigma(1),\ldots,\sigma(i)$ were defined in such a way that $|C_1(\sigma,i) - C_2(\sigma,i)|\leq 
n/\hat\ell+1$, where $i\leq\hat\ell-1$.

If $C_1(\sigma,i)=C_2(\sigma,i)$, then clearly $\sigma(i+1)$ can be defined as being any of the remaining integers on 
$[\hat\ell]$.\ So, w.l.o.g.\ 
assume that $C_1(\sigma,i) = C_2(\sigma,i) + k$, with
${1\leq k\leq n/\hat\ell+1}$. But since $C_1(\sigma,i) + C_2(\sigma,i) \leq i(n/\hat\ell + 1)$, we can conclude that
\begin{equation}\label{eq:C2i}
C_2(\sigma,i) \leq \frac{in}{2\hat\ell} - \frac{k-i}{2}.
\end{equation}

We will prove that there exists some $r\in[\hat\ell] \setminus \bigcup_{j=1}^i \sigma(j)$ with ${C_2(I_{r})\geq k/2}$. 
Suppose by contradiction that 
$C_2(I_{r})<k/2$ for all integers ${r\in[\hat\ell] \setminus \bigcup_{j=1}^i \sigma(j)}$. This fact together with 
\eqref{eq:C2i} implies the 
following.
\begin{equation}\label{eq:C2l}
\begin{aligned}
C_2(W) &\leq C_2(\sigma,i) + (\hat\ell - i)\frac{k}{2}=\frac{in}{2\hat\ell} + (\hat\ell-i-1)\frac{k}{2} + \frac{i}{2}\\
&\leq\left(\frac{\hat\ell-1}{\hat\ell}\right)\frac{n}{2}+\frac{\hat\ell}{2}=\left(\frac{n(\hat\ell-1)+\hat\ell^{\,2}}{n\hat\ell}\right)\frac{n}{2},
\end{aligned}
\end{equation}
where the last inequality holds because $k\leq n/\hat\ell + 1$ and $i\leq \hat\ell-1$.

Since $C_1(W) + C_2(W) = n$, using \eqref{eq:C2l} we know that 
\begin{equation*}
\frac{C_1(W)}{C_2(W)} = \frac{n}{C_2(W)} - 1\geq 1 + \frac{2(n-{\hat\ell}^{\,2})}{n(\hat\ell-1)+{\hat\ell}^{\,2}}
>1+\beta,
\end{equation*}
where the last inequality follows by the choice of $n_0$, because $\beta\leq 2/\hat\ell$. But this contradicts the 
$\beta$-balancedness of the 
coloring $\chi$ of $W$. Therefore, there
exists $r\in[\hat\ell] \setminus \bigcup_{j=1}^i \sigma(j)$ with~$C_2(I_{r})\geq k/2$. Set $\sigma(i+1)=r$. Then
\begin{equation*}
\begin{aligned}
C_1(\sigma,i+1) &= C_1(\sigma,i) + C_1(I_{r})\leq \left(C_2(\sigma,i) + k\right) + \left(\frac{n}{\hat\ell} +1 - \frac{k}{2}\right)\\
&=\left( C_2(\sigma,i) + \frac{k}{2}\right) + \frac{n}{\hat\ell}+1\leq C_2(\sigma,i+1) + \frac{n}{\hat\ell}+1.
\end{aligned}
\end{equation*}
Since $C_1(\sigma, i+1)\geq C_2(\sigma,i+1) - (n/\hat\ell+1)$, we conclude from the lat inequality that 
$|C_1(\sigma,i+1) - C_2(\sigma,i+1)|\leq 
n/\hat\ell+1$.
\end{proof}

Let $H=(W,E)$ be a graph with $W=\{w_1,\ldots,w_n\}$ and let $\chi\colon W\to[2]$ be a coloring of $W$. Consider a 
partition of $W$ into a set of 
intervals
$\mathcal{I}=\{I_1,\ldots,I_{\hat\ell}\}$. We define $C_i(I,\sigma,a,b)=\sum_{j=a}^{b} C_i(I_{\sigma(j)})$ for $i=1,2$. 
If it is clear what partition 
we are considering then we write $C_i(\sigma,a,b)$
for simplicity.

\begin{corollary}\label{cor:cor}
For every integer $\hat \ell\geq 1$ there exists $n_0$ such that if $H=(W,E)$ is a graph with $W=\{w_1,\ldots,w_n\}$ 
with $n\geq n_0$, then every 
$\beta$-balanced $2$-coloring $\chi$
of $W$ with 
$\beta\leq 2/\hat\ell$, and every partition of $W$ into intervals ${I_1,\ldots,I_{\hat\ell}}$ with sizes 
${|I_1|\leq\ldots \leq|I_{\hat\ell}|\leq |I_1|+1}$ there 
exists a permutation
$\sigma\colon[\hat\ell]\to[\hat\ell]$ such that for every pair of integers $1\leq a<b\leq \hat\ell$,
\begin{equation}
|C_1(\sigma,a,b) - C_2(\sigma,a,b)|\leq 2\left(\frac{n}{\hat\ell}+1\right).
\end{equation}
\end{corollary}

\begin{proof}
Fix $\hat\ell\geq 1$ and let $n_0$ be obtained from Lemma \ref{lemma:balCol} applied with $\hat\ell$.  Let $H=(W,E)$ be 
a graph with 
$W=\{w_1,\ldots,w_n\}$ with $n\geq n_0$. Now fix a \text{$\beta$-balanced} $2$-coloring $\chi$ of~$W$ and a partition 
of $W$ into intervals 
$I_1,\ldots,I_{\hat\ell}$ with ${|I_1|\leq\ldots \leq|I_{\hat\ell}|\leq |I_1|+1}$, where $\beta\leq 2/\hat\ell$.

Let $\sigma$ be the permutation given by Lemma \ref{lemma:balCol}. Fix integers $1\leq a<b\leq \hat\ell$ and suppose 
w.l.o.g.\ that 
$C_1(\sigma,a,b)\geq C_2(\sigma,a,b)$. Therefore
\begin{equation*}
\begin{aligned}
C_1(\sigma,a,b) &= C_1(\sigma,b) - C_1(\sigma,a-1)\\
&\leq (C_2(\sigma,b) + n/\hat\ell + 1) - (C_2(\sigma,a-1) - (n/\hat\ell+1))\\
&=C_2(\sigma,a,b) + 2(n/\hat\ell + 1).
\end{aligned}
\end{equation*}
\end{proof}

The next result, the main result of this subsection, guarantees the local balancedness that we need.

\begin{lemma}\label{lemma:balancedVectors}
For every $\xi>0$ and every integer $\hat \ell\geq 1$ there exists $n_0$ such that if $H=(W,E)$ is a graph with 
$W=\{w_1,\ldots,w_n\}$ with $n\geq 
n_0$, then for every \text{$\beta$-balanced} $2$-coloring $\chi$ of $W$ with
$\beta\leq 2/\hat\ell$, and every partition of $W$ into intervals $I_1,\ldots,I_{\hat\ell}$ with $|I_1|\leq\ldots 
\leq|I_{\hat\ell}|\leq |I_1|+1$ there 
exists a permutation
$\sigma\colon[\hat\ell]\to[\hat\ell]$ such that for every pair of integers $1\leq a<b\leq \hat\ell$ with $b-a\geq 
7/\xi$, we have
\begin{equation*}
|C_1(\sigma,a,b)-{C_2(\sigma,a,b)}|\leq \xi {C_2(\sigma,a,b)},
\end{equation*}
\end{lemma}

\begin{proof}
Fix constants $\xi>0$, $\hat\ell\geq 1$ and let $n_0'$ be obtained by Corollary \ref{cor:cor} applied with $\hat\ell$.  Let 
$H=(W,E)$
be a graph with vertex set $W=\{w_1,\ldots,w_n\}$ for  
\[
	n\geq n_0=\max\{n_0',((4+2\xi)/(3-2\xi))\hat\ell\}
\] 
and fix a 
$\beta$-balanced $2$-coloring $\chi$ of $W$ and a partition 
of $W$ into intervals $I_1,\ldots,I_{\hat\ell}$ with
$|I_1|\leq\ldots \leq|I_{\hat\ell}|\leq |I_1|+1$ where $\beta\leq 2/\hat\ell$.

Let $\sigma$ be the permutation given by Corolary~ \ref{cor:cor}. Fix integers $1\leq a<b\leq \hat\ell$ such that 
$b-a>7/\xi$. Note 
that, by Corollary \ref{cor:cor},
\begin{equation}\label{eq:cameFromCor}
|C_1(\sigma,a,b) - C_2(\sigma,a,b)|\leq 2(n/\hat\ell + 1). 
\end{equation}
The above inequality and the fact that $C_1(\sigma,a,b) + C_2(\sigma,a,b) \geq (b-a)(n/\hat\ell)$ implies
\begin{equation*}
 C_2(\sigma,a,b)\geq \left(\frac{b-a}{2}\right)\frac{n}{\hat\ell} - (n/\hat\ell + 1).
\end{equation*}
By the choice of $a$, $b$ and $n_0$, we have
\begin{equation}\label{eq:eq1}
 C_2(\sigma,a,b)\geq (2/\xi)(n/\hat\ell + 1).
\end{equation}
Putting inequalities \eqref{eq:cameFromCor} and \eqref{eq:eq1} together we conclude the proof.
\end{proof}

\section{Proof of the main result}\label{sec:main}

Before going into the details of the proof of Theorem~\ref{thm:threecolor} we give some brief overview discussing the 
main ideas of the proof and 
explaining how to connect the results of Section~\ref{sec:aux}.

\subsubsection*{Overview of the proof of Theorem \ref{thm:threecolor}}
For every $\gamma>0$ and sufficiently large $n$, given an arbitrary edge coloring of $K_N$ with $3$ colors for 
$N={(2+\gamma)n}$ we want to prove that if $H$ is a 
$(\beta,\Delta)$-balanced graph on $n$ vertices,
then we always find a monochromatic copy of $H$ in $K_N$.

The strategy to prove Theorem \ref{thm:threecolor} is to apply the Embedding Lemma (Lemma~\ref{lemma:GeneralEmbedding}) 
to find the desired copy of 
$H$ in $K_N$. In order to do
this we use Lemma \ref{lemma:FindTheNiceTree} to find a monochromatic subgraph $G$ of $K_N$ composed of sufficiently 
dense regular pairs. So, using 
Facts
\ref{fact:Slicing} and \ref{fact:SuperSlicing} it is easy to see that deleting some vertices of $G$ we can find a 
monochromatic graph $G'\subset G$ 
which has a regular partition containing
super-regular pairs covering $(1+o(1))n$ vertices.

In the second part of the proof we carefully construct a partition of 
$V(H)$ and, since~$H$ has small bandwidth, we make use of 
Lemma~\ref{lemma:balancedVectors} to show that this 
partition is compatible with the partition of $G'$. Then, we can apply the Embedding Lemma to find the monochromatic 
copy of $H$, concluding the 
proof.

\subsubsection*{Proof of Theorem \ref{thm:threecolor}}
Let $\gamma>0$ and $\Delta\geq 1$ be given. Lemma \ref{lemma:FindTheNiceTree} applied with $\gamma$ gives~$\eps_0$. 
Next we apply Lemma 
\ref{lemma:GeneralEmbedding} with $d=1/4$ and $\Delta$ to get $\eps_1$. Set 
$$
\eps=\min\{\eps_0,\eps_1/2, \gamma/19\}.
$$
Since $\eps\leq \eps_0$, Lemma~\ref{lemma:FindTheNiceTree} gives to us a natural number $K_0$. Fix $\xi=\gamma/304$ and 
let $n_0$ be obtained by 
an application of Lemma~\ref{lemma:balancedVectors} with parameters $\xi$ and $K_0$. Set
\begin{equation*}
\beta=\eps\xi(1+2\xi)/36\Delta^2K_0^2.
\end{equation*}

Let $H=(W,E_H)$ be a balanced $(\beta,\Delta)$-graph on $n$ vertices. Now put ${N=\lfloor(2+\gamma)n\rfloor}$, where 
$N\geq \max\{n_0,K_0\}$. 
Consider an arbitrary coloring $\chi_{K_N}\colon E(K_{N})\to[3]$ of the edges of $K_N$. We want to
show
that every such coloring yields a monochromatic copy of $H$.

\medskip
\noindent\textit{Partitioning the vertices of $K_{N}$}.
Next we find a monochromatic and sufficiently regular subgraph $G'$ of $K_{N}$. By Lemma~\ref{lemma:FindTheNiceTree}, 
there are a color 
(say color 1), integers $\ell,\ell',k$ with ${\ell,\ell'\leq k \leq K_0}$ and $\ell\geq (1-\gamma/4)k/4$, a tree $T$ on 
vertex set
$\{x_1,\dots,x_\ell,y_1,\dots,y_\ell,z_1,\dots,z_{\ell'}\}$ containing
a matching $M$ with edge set $E_M =\{ x_i y_i\colon i=1,\ldots,\ell\}$ with an even distance in $T$
between any $x_i$ and $x_j$ for all $i$ and $j$, such that there exists a partition $(V_i)_{i\in[k]}$ of ${V=V(K_N)}$ 
such that $K^1_N$ is 
$(\eps,1/3)$-regular on $T$ and $|V_1|=\ldots=|V_k|=m$, where $m\geq (1-\eps)N/k$. Let $G_T$ be the subgraph of $K^1_N$ 
induced by the classes in 
$(V_i)_{i\in[k]}$ corresponding to the vertices of $T$.

In order to apply the Embedding Lemma, we need the classes of $G_T$ that correspond to the matching edges to form 
super-regular pairs and the other 
pairs of classes should be sufficiently regular. We
can ensure this by deleting some vertices of $G_T$. In fact, applying Fact \ref{fact:SuperSlicing} and, after that, 
Fact \ref{fact:Slicing}, it is 
easy to see that we find a subgraph $G'\subset G_T$ with
classes 
$A_1,\dots,A_\ell,$ $B_1,\dots,B_\ell,$ $C_1,\dots,C_{\ell'}$ of size at least $(1-\eps)m$ corresponding, respectively, 
to the vertices 
$x_1,\dots,x_\ell,y_1,\dots,y_\ell,z_1,\dots,z_{\ell'}$ of
the
tree $T$, such that the bipartite graphs induced by $A_i$ and $B_i$ are $(2\eps, 1/3-\eps)$-super-regular and the 
bipartite graphs induced by all the 
other pairs are $(2\eps, 1/3-\eps)$-regular.
Furthermore, let $D_{\min}$ be the set with the smallest cardinality among the sets in $A_1,\dots,A_\ell,$ 
$B_1,\dots,B_\ell,$ 
$C_1,\dots,C_{\ell'}$. Since  $\eps\leq \gamma/19$, ${N=\lfloor(2+\gamma)n\rfloor}$, $m\geq
(1-\varepsilon)N/k$ and $\ell\geq (1-\gamma/4)k/4$, one can see that 
\begin{equation}\label{eq:limitD}
|D_{\min}|\geq (1+\gamma/152)n/2\ell. 
\end{equation}

\medskip
\noindent\textit{Partitioning the vertices of $H$}.
Now it is time to construct a partition of $W$ ready for the application of Lemma~\ref{lemma:GeneralEmbedding}.  Since 
$H$ is a balanced 
$(\beta,\Delta)$-graph, there exists a coloring ${\chi_H\colon
V(H)\to[2]}$ such that ${\big||\chi^{-1}(1)|-|\chi^{-1}(2)|\big|\leq \beta|\chi^{-1}(2)|}$.

Let $w_1,\ldots,w_n$ be an ordering of $W$ such that $|i-j|\leq \beta n$ for every $w_i w_j\in E_H$ and let $\hat\ell$ 
be the smallest integer 
dividing $n$ with ${\hat\ell\geq (7 K_0/\xi) + \ell\geq \ell(7/\xi + 1)}$. Consider  the
partition of $V(H)$ into intervals $I_1,\ldots,I_{\hat\ell}$ with ${|I_1|=\ldots= |I_{\hat\ell}|=n/\hat\ell}$ taking 
this ordering into account, i.e., 
$I_i=w_{(i-1)n/\hat\ell + 1},\ldots,w_{in/\hat\ell}$
for ${i=1,\ldots,\hat\ell}$. By Lemma \ref{lemma:balancedVectors}, since $\beta\leq 2/\hat\ell$, there exists a 
permutation 
$\sigma\colon[\hat\ell]\to[\hat\ell]$ such that
\begin{equation*}
|C_1(\sigma,a,b)-{C_2(\sigma,a,b)}|\leq \xi {C_2(\sigma,a,b)}
\end{equation*}
for all integers $1\leq a<b\leq \hat\ell$ with $b-a\geq 7/\xi$. Define ${a_i=(i-1)\hat\ell/\ell + 1}$ and 
${b_i=i\hat\ell/\ell}$ and 
consider the blocks
$J_i=\{I_{\sigma(a_i)},I_{\sigma(a_i+1)}, \ldots, I_{\sigma(b_i)}\}$ for ${i=1,\ldots,\ell}$. We write $C_1(J_i)$ for 
$C_1(\sigma,a_i,b_i)$ and $C_2(J_i)$ for $C_2(\sigma,a_i,b_i)$.
Thus, for ${i=1,\ldots,\ell}$, since $b_i - a_i = \hat\ell/\ell + 1\geq 7/\xi$, we have
\begin{equation}\label{eq:equalColors}
|C_1(J_i)-{C_2(J_i)}|\leq \xi {C_2(J_i)},
\end{equation}

Recall we have found a tree $T$ on vertex set $\{x_1,\dots,x_\ell,y_1,\dots,y_\ell,z_1,\dots,z_{\ell'}\}$ containing
matching edges $E_M =\{ x_i y_i\colon i=1,\ldots,\ell\}$ such that the distance in $T$ between any~$x_i$ and $x_j$ for 
all $i$ and $j$ is even. Our
partition of $W$ will be composed of clusters $X_1,\dots,X_{\ell},Y_1,\dots,Y_{\ell},Z_1,\dots,Z_{\ell'}$ 
corresponding to 
$x_1,\dots,x_\ell,y_1,\dots,y_\ell,z_1,\dots,z_{\ell'}$.

For every $i=1,\ldots,\ell$, we  will put most of the vertices of $J_i$ in the clusters $X_i$ and~$Y_i$, depending on 
the color they received from 
$\chi_H$. The remaining vertices will be distributed
in order to make it possible to ``walk'' between the matching clusters.

We divide each interval $I_i$ in two parts. The first one, called \emph{link} of $I_i$, is denoted by~$L_i$. The links 
are responsible to 
make the connections between the
matching clusters. For the last interval, we set
$L_{\hat\ell}=\emptyset$. For $1\leq i\leq \hat\ell-1$, if $I_i$ and $I_{i+1}$ are in the same block $J_r$, then 
$L_i=\emptyset$.

Suppose that $I_i\in J_r$ and $I_{i+1}\in J_{s}$ with $r\neq s$ and $1\leq i\leq \hat\ell-1$. Let
$P_T(r,s)$ be the path of $T$ between $x_r$ and
$x_s$ and consider the path $P^{\int}_T(r,s)\subset P_T(r,s)$ obtained by excluding the vertices of the set 
$\{x_r,y_r,x_s,y_s\}$ from $P_T(r,s)$, 
i.e., 
$P^{\int}_T(r,s)$ is the ``internal'' part of the
path of $T$ that one should use to reach $x_s$ from $x_r$. For a lighter notation set~$t_{r,s}=|P^{\int}_T(r,s)|$.
We divide the ${(t_{r,s}+1)\beta n}$ last vertices of $I_i$ in $t_{r,s}+1$ ``pieces'' of size~$\beta n$, 
respecting their sequence in the 
interval, where the $j$-th piece is denoted by $L_i(j)$ for ${1\leq j\leq t_{r,s}+1}$, that is,
\begin{equation*}
L_i(j)=w_{(i-(t_{r,s}+2-j)\beta \hat\ell)n/\hat\ell+1},\ldots,w_{(i-(t_{r,s}+1-j)\beta \hat\ell)n/\hat\ell}.
\end{equation*}
We put $L_i = \{L_i(1),\ldots,L_i(t_{r,s}),L_i(t_{r,s}+1)\}$.

Since we have described the links, we can now define the main part of the intervals. We define $\KE_i = I_i \setminus 
L_i$ as the \emph{kernel} of 
the interval $I_i$, which will be placed on the matching clusters $X_i$ and $Y_i$.

We have to construct the clusters that will compose the partition of $H$. Initially, let each cluster in 
$\{X_1,\dots,X_{\ell},Y_1,\dots,Y_{\ell},Z_1,\dots,Z_{\ell'}\}$ be empty. Consider the
block $J_i$ for every $1\leq i\leq \ell$. For each interval $I_p\in J_i$ we include in $X_i$ all the vertices $w$ of 
the kernel~$\KE_p$ with 
$\chi_H(w)=1$ and we include in $Y_i$ all the vertices $w$ of $\KE_p$ with $\chi_H(w)=2$.

The next step is to accommodate all the links. Consider the interval $I_i$ for ${1\leq i\leq \hat\ell-1}$ and assume 
that $I_i$ 
is in $J_r$ and $I_{i+1}$ is in $J_{s}$ with $r\neq
s$, otherwise the link we are looking for is empty and there is nothing to do. Denote the internal path 
$P^{\int}_T(r,s)$ of $P_T(r,s)$ by 
$\{u_1,\ldots,u_{t_{r,s}}\}$ and let $u_0$ and $u_{t_{r,s}+1}$ be, respectively, 
the
vertices of $T$ connected to $u_1$ and~t$u_{t_{r,s}}$ in $P_T(r,s)$. 

Now we will show how it is possible to ``walk'' between the matching clusters. note that $u_0$ can be either $x_r$ or 
$y_r$. Without loss of generality we assume 
that $u_0=x_r$. For $1\leq j\leq t_{r,s}+1$, we put the vertices $w$ of $L_i(j)$ with $\chi_H(w)=1$ in the 
corresponding class of $u_{j-1}$ if $j$ is 
even, and in the corresponding class of $u_j$ if $j$ is odd. For those $w$ with $\chi_H(w)=2$, we do the other way 
around, i.e., we put them in the 
corresponding class of $u_{j}$ if $j$ is even, and in the corresponding class of 
$u_{j-1}$ if $j$ is odd. Since $x_i$ and $x_j$ are at an even distance for all $1\leq i<j\leq \ell$ and the links have 
size $\beta n$, we know 
that there is no edges inside the clusters and if there is an edge between two
clusters, then the corresponding edge is present in $T$.

\medskip
\noindent\textit{Applying the Embedding Lemma}.
Here we will show that the vertex partition of $W$ is $(2\eps_1,T,M)$-compatible with the partition of $V(G')$ we 
constructed before. Thus, we can 
apply the Embedding Lemma to find the desired monochromatic copy of $H$ in $K_N$.

The first step is to bound by above the size of each cluster in the partition 
\[
	\{X_1,\dots,X_{\ell},Y_1,\dots,Y_{\ell},Z_1,\dots,Z_{\ell'}\}
\] of
$W$. Note that, for every $1\leq i\leq \ell$, we have
${C_1(J_i)+C_2(J_i) = n/\ell}$. Using this fact and~\eqref{eq:equalColors} one can easily obtain that, for every $1\leq 
i\leq \ell$,
\begin{equation}
(1-\xi)\frac{n}{2\ell} \leq C_1({J_i}), C_2({J_i}) \leq (1+\xi)\frac{n}{2\ell}.
\end{equation}

By the construction, every set $X_i$ (resp.\ $Y_i$) is composed only of vertices $v$ with $\chi(v)=1$ ($\chi(v)=2$). Furthermore, 
these vertices can come from one kernel and at most two pieces of each link. Then, 
\begin{equation}\label{eq:limitXY}
|X_i|,|Y_i| \leq (1+\xi)\frac{n}{2\ell} + 2\hat\ell \beta n
=\left(1+\xi+  4\ell\hat\ell\beta\right) \frac{n}{2\ell}
\leq |D_{\min}|,
\end{equation}
where the last inequality follows by inequality \eqref{eq:limitD} and the choice of $\xi$, $\beta$ and $\hat\ell$.

For the clusters $Z_i$, for $1\leq i\leq \ell'$, we know that they are composed only of vertices in at most two pieces 
of each link. Thus,
\begin{equation}\label{eq:limitZ}
|Z_i| \leq 2\hat\ell\beta n=(4\ell\hat\ell\beta)\frac{n}{2\ell}\leq \frac{\eps}{\Delta^2}|D_{\min}|,
\end{equation}
where the last inequality follows by inequality \eqref{eq:limitD} and the choice of $\beta$ and $\hat\ell$.

Now we can check that the partitions of $W$ and $V(G')$ are compatible. Based on Definition 
\ref{def:compatiblePartitions} we define the sets $U_j$ 
and $U'_j$ for $1\leq j\leq 2\ell+\ell'$ with respect to the partition
$\{X_1,\dots,X_{\ell},Y_1,\dots,Y_{\ell},Z_1,\dots,Z_{\ell'}\}$ of $W$. Define $W_j=X_j$ if $1\leq j\leq \ell$, 
$W_j=Y_{j-\ell}$ if $\ell+1\leq j\leq 
2\ell$, and $W_j=Z_{j-2\ell}$ if $2\ell+1\leq
j\leq 2\ell+\ell'$. Then, we will verify that the four conditions of Definition \ref{def:compatiblePartitions} hold:

\begin{enumerate}[label=\rmlabel]
 \item By the construction of the partition of $W$, if there is an edge between two clusters, then the 
corresponding edge is present in 
$T$.

\item Owing to \eqref{eq:limitD} every set $D$ in the partition
$\{A_1,\dots,A_\ell,$ $B_1,\dots,B_\ell,$ $C_1,\dots,C_{\ell'}\}$ of~$V(G')$ has size $|D|\geq
(1+\gamma/152)n/2\ell$. So, inequalities
\eqref{eq:limitXY} and \eqref{eq:limitZ} show that condition~\ref{d:cPii} holds.

\item Fix $1\leq j\leq 2\ell+\ell'$. Define $U_j$ as the set of vertices of~$W_j$ with neighbors in some~$W_k$ 
with $j\neq k$ and 
$\{j,k\}\notin M$. We divide in two cases:
\begin{enumerate}[label=\alabel]
\item  $2\ell+1\leq j\leq 2\ell+\ell'$: We have $U_j=Z_{j-2\ell}$. By \eqref{eq:limitZ}, $|U_j|\leq \eps 
|D_{\min}|/\Delta$.
\item $1\leq j\leq 2\ell$: In this case, $U_j$ is composed only of neighbors of vertices in exactly one set of 
$\{Z_1,\ldots,Z_{\ell'}\}$. Thus, since $\Delta$ if the maximum degree of $H$, by
\eqref{eq:limitZ}, we conclude that $|U_j|\leq \eps |D_{\min}|/\Delta$.
\end{enumerate}
 Thus, for every $j=1,\ldots,2\ell+\ell'$ we have
\begin{equation}\label{eq:UjDmin}
|U_j|\leq \frac{\eps}{\Delta} |D_{\min}|,
\end{equation}
which shows that condition~\ref{d:cPiii} holds. 

\item Define the set $U'_j=N_H(U)\cap (W_j\setminus U)$, where $U=\bigcup_{i=1}^{2\ell+\ell'}U_i$. Consider the 
following cases.
\begin{enumerate}[label=\alabel]
\item $2\ell+1\leq j\leq 2\ell+\ell'$: Note that since every vertex of $Z_{j-2\ell}$
belongs to $U_j$, we have $U'_j=\emptyset$. Thus, it is obvious that $|U'_j|\leq 
|D_{\min}|$.
\item $\ell+1\leq j\leq 2\ell$: Here, $U'_j\subset W_j= Y_{j-\ell}$. Then, $U'_j$ is composed only of neighbors of 
$U_{j-\ell}\subset 
X_{j-\ell}$. Then, using
\eqref{eq:UjDmin}, we have ${|U'_j|\leq \Delta|U_{j-\ell}|\leq \eps |D_{\min}|}$.
\item $1\leq i\leq \ell$: This case is analogous to case (b).
\end{enumerate}

\end{enumerate}

Since we proved that the four conditions of Definition \ref{def:compatiblePartitions} hold, the partition 
\[
	\{X_1,\dots,X_{\ell},Y_1,\dots,Y_{\ell},Z_1,\dots,Z_{\ell'}\}
\]
of $W$ is $(2\varepsilon,T,M)$-compatible (then, it is clearly $(\varepsilon_1,T,M)$-compatible) with 
\[
	\{A_1,\dots,A_\ell,B_1,\dots,B_\ell,C_1,\dots,C_{\ell'}\}\,,
\] 
which is a partition of $V(G')$. Then, by Lemma 
\ref{lemma:GeneralEmbedding}, we conclude that 
$H\subset G'$. 
This finishes the proof, since $G'$ is a
monochromatic subgraph of $K_{N}$.
\qed

\section{Sketch of the proof of Theorem \ref{thm:twocolor}}\label{sec:two}

We show that for every $\gamma>0$ and natural number $\Delta$, there exists a
constant $\beta>0$ such that for every sufficiently large $(\beta,\Delta)$-graph $H$ with a proper
2-coloring $\chi_H\colon\, V(H) \to [2]$ where $t_1=|\chi_H^{-1}(1)|$ and
$t_2= |\chi_H^{-1}(2)|$, with ${t_1\leq t_2}$, we can find a monochromatic copy of~$H$ in every edge coloring of $E(K_{N})$ with ${N=(1+\gamma) \max 
\{2t_1+t_2, 2t_2 \}}$. Let~$H$ be such a graph and assume
$2t_1\geq t_2$ (the complementary case can be solved in a similar way).

The proof of Theorem \ref{thm:twocolor} is very similar to the proof of Theorem \ref{thm:threecolor}. Here we also embed~$H$ in parts, considering a partition 
of a monochromatic subgraph $G$ of $K_N$. The partition we need is composed of a special cluster $W$ and clusters 
$X_1,Y_1,\ldots,X_m,Y_m$ corresponding to a 
``large'' matching $M$ with matching edges ${E_M=\{x_i,y_i\colon i=1,\ldots,m\}}$ such that for 
every~$i=1,\ldots,m$ the pairs $\{X_i,Y_i\}$ are super-regular and the pairs $\{X_i,W\}$ 
are regular. The special cluster $W$ is needed to allow us to ``walk'' between the clusters 
$X_1,Y_1,\ldots,X_m,Y_m$.

The problem in the preparation of the host monochromatic graph $G$ is the fact that $H$ is not as balanced as it is in the setup of Theorem \ref{thm:threecolor}. So,
in order to embed $H$ in $G$ we need that $|Y_i|/|X_i|=t_2/t_1$. Fortunately, by~\cite{HaLuTi02}*{Theorem~3}, since 
$t_2/t_1\leq 2$ in the case we are considering,
we can find such a monochromatic graph $G$. Using Fact \ref{fact:SuperSlicing} we
can easily make the matching pairs super-regular.

Now we have to prepare the graph $H$ for the embedding. We consider the ordering of its vertices respecting the bandwidth condition and divide the set of 
vertices into intervals. Thus, we can find a permutation of such intervals
such that blocks of intervals fit into the super-regular pairs of $G$. Then, using few vertices we can ``walk'' from one super-regular pair to 
another as
done in the proof of Theorem~\ref{thm:threecolor} and we are done.

\end{document}